\documentclass[reqno, 12pt]{amsart}

\usepackage{amsmath}    
\usepackage{amsthm}     
\usepackage{graphicx}   
\usepackage{verbatim}   
\usepackage{color}      
\usepackage{subfigure}  
\usepackage{hyperref}   
\usepackage[utf8]{inputenc}

\newtheorem{theorem}{Theorem}[section]

\newtheorem{lemma}[theorem]{Lemma}
\newtheorem{corollary}[theorem]{Corollary}
\theoremstyle{definition}
\newtheorem{definition}[theorem]{Definition}

\theoremstyle{remark}

\newcommand{\R}{\mathbb{R}}

\title{Superbridge and Bridge Indices for Knots}
\author[Adams et al]{Colin Adams, Nikhil Agarwal, Rachel Allen, Tirasan Khandhawit, Alex Simons, Rebecca Winarski, and Mary Wootters}
\date{\today}
\begin{document}

\setlength{\parindent}{0pt}
\setlength{\parskip}{8pt}

\begin{abstract}
We improve the upper bound on superbridge index $sb[K]$ in terms of bridge index $b[K]$ from $sb[K] \leq 5b -3$ to $sb[K]\leq 3b[k] - 1$. 
\end{abstract}
\maketitle

\section{Introduction}

In a seminal paper \cite{Kuip}, N. Kuiper introduced superbridge index for knots,  a variation of the better-known bridge index, first introduced by Schubert in \cite{Sch}. Let $K$ be a particular embedding of a knot in 3-space, which we will refer to as a conformation, and let $[K]$ denote the set of all conformations that are equivalent to it, generating the same knot type. Letting $\vec{v}$
represent a unit vector giving a direction in 3-space to which we will project the knot, we can define bridge index as
follows.

 \begin{definition} The {\bf bridge index} of a knot $[K]$ is given by 
      
     $$b[K] = \underset{K \in [K]}{\text{min}}\: \underset{\vec{v} \in S^2}{\text{min}} (\text{\# of local maxima of $K$ in direction $\vec{v}$})$$.  \end{definition}
  
Given this formulation of bridge index, it is simple to give Kuiper's variant:
     
     \begin{definition} The {\bf superbridge index} of a knot is given by
     
     $$sb[K] = \underset{K \in [K]}{\text{min}}\: \underset{\vec{v} \in S^2}{\text{max}} (\text{\# of local maxima of $K$ in direction $\vec{v}$})$$.
     
     \end{definition}

It is obvious from the definition that $sb[K] \geq b[K]$. In fact, in \cite{Kuip}, Kuiper proved
that $sb[K] > b[K]$ for any nontrivial knot. Superbridge index is related to several
other invariants.

\begin{definition} The {\bf geometric degree of a knot conformation $K$} is the greatest
number of times that a plane intersects the knot conformation, denoted d(K).
The {\bf geometric degree of a knot type [K]} is given by 

$$d[K] = \underset{K \in [K]}{\text{min}} d(K)$$ 
\end{definition} 

Note that $d[K]$ is always even since if a plane is tangent to an embedding of the knot $K$, we can move
the plane slightly to obtain one fewer intersections, and otherwise, intersections
pair up according to how they are connected by the knot to one side of the plane.
Each such pair creates at least one local maximum in the normal direction to the
plane. This also demonstrates the following useful result.

\begin{lemma}
\label{degreesb} $d[K] \leq 2sb[K]$.
\end{lemma}

One of the reasons that superbridge index is interesting is its relationship with one of the most natural invariants for knots.

\begin{definition} The {\bf stick index} of a knot type, denoted $s[K]$,  is the least number of sticks glued end-to-end to obtain a conformation of that knot type.
\end{definition}
In \cite{Jin1}, Jin noted the following.

\begin{lemma}
\label{sticksb} $sb[K] \leq s[K]/2$.
\end{lemma}

\begin{proof}
Choose a stick conformation that realizes the stick number. Then for any choice of a direction vector $\vec{v} \in S^2$, the maxima can only occur at vertices or along entire edges. Since for every maximum, there must be a corresponding minimum, the superbridge number of this conformation is at most $s[K]/2$. Therefore the superbridge index, which is the minimum over all conformations, is also bounded above by $s[K]/2$.
\end{proof}

Information about superbridge index has been very useful in determining stick index as in \cite{Jin1, Jin2, Jin3}.

In \cite{Kuip}, Kuiper determined the geometric degree of all torus knots, denoted $T_{p,q}$
with $p < q$, and then used Lemma \ref{degreesb} together with upper bounds to determine
superbridge index for all torus knots as well:

\begin{theorem}For $p < q$,  $d(T_{p,q}) = \text{min}\{4p, 2q\}$ and $sb(T_{p,q}) = \text{min}\{2p, q\}$.
\end{theorem}

In the same paper, Kuiper also
proved that $sb[K] \leq 2\beta[K]$, where $\beta [K]$ is the braid index of $[K]$. In \cite{FLS}, it was
proved that $sb[K] \leq 5b[K] - 3$. Here, we obtain the following improvement of that upper bound. 

\begin{theorem} \label{mainthm}$sb[K] \leq  3b[K] -1$. 
\end{theorem}

Thus, we now know bridge index sandwiches superbridge index via  $b[K] + 1 \leq sb[K] \leq 3b[K] -1$.

 Theorem \ref{mainthm} implies the superbridge index of 
any 2-bridge knot is either 3, 4 or 5.  In \cite{Milnor}, Milnor proves that
every nontrivial knot has geometric degree at least 6. Hence Lemma \ref{degreesb} yields:

\begin{corollary} Any two-bridge knot has geometric degree 6, 8 or 10. 

\end{corollary}

In \cite{JJ2}, the authors use quadrisecants to
show that there are only finitely many 3-superbridge knots, all of them in the list $3_1, 4_1, 5_2, 6_1, 6_2, 6_3, 7_2,$ $7_3, 7_4, 8_4, 8_7$, and $8_9$. In this list, $3_1$ and $4_1$ are known to be
3-superbridge knots, and therefore knots of geometric degree 6.  Jeon and Jin conjecture that $3_1$ and $4_1$ are the only
3-superbridge knots. 

\begin{corollary} Every 2-bridge knot other than $3_1, 4_1, 5_2, 6_1, 6_2, 6_3, 7_2,$ $7_3, 7_4, 8_4, 8_7$, and $8_9$ has superbridge index 4 or 5.

\end{corollary}

\section{Upper Bound on Superbridge Index}

In this section, given a knot type $[K]$, we give a construction of a conformation that will be used to prove our main theorem: $sb[K] \leq 3b[K] -1.$

Let $\vec{v} \in \R^3$ be a unit vector with $\vec{v} = v_1\hat{i}+v_2\hat{j}+v_3\hat{k}$ and let $\eta(t) = (\cos  t, \sin  t, \cos^2  t).$ 

Kuiper was able to prove that the superbridge index of a knot, $sb[K]$, is bounded above by twice the braid index of that knot $\beta[K]$ by taking a braid conformation of the knot that follows the curve $\eta(t)$.  Each string  contributes either one or two local maxima given any direction defined by $\vec{v}$. We are adapting Kuiper's argument, but using a conformation of a knot that realizes bridge index (instead of braid index), and then placing it on the same curve $\eta(t)$ to show that $sb[K] \leq 3b[K] - 1$, where $b[K]$ is the bridge index of a knot K. 

Since we rely heavily on Kuiper's argument, we will summarize the argument here \cite{Kuip}:

Kuiper first notes that the curve $\eta(t)$ has at most two maxima in any direction. That is, it has superbridge number equal to 2. He then takes a circular \textit{r-braid} knot and parametrizes it for a small $\epsilon > 0$ by $$\lambda_{\epsilon}(t) = (cos (rt) * (1+\epsilon\lambda_1(t)), sin (rt) * (1+\epsilon\lambda_1(t)), cos^2 (rt) + \epsilon\lambda_2(t)), $$
in $t$ modulo $2\pi$, where $\lambda_1^2 + \lambda_2^2 \leq 1$. He approximates $\lambda_1(t)$ and $\lambda_2(t)$ by finite linear expressions in cos $n_j(t)$ and sin $n_j(t)$ for $n_j \in \mathbb{N}, j \in \mathbb{Z}$ 
so that we have finite polynomials in cos$(t)$ and sin$(t)$. This creates a conformation of a knot isotopic to the original \textit{r-braid} knot that lives inside a torus within the $\epsilon$-neighborhood of $\eta(t)$. Recalling cos$^2 (t) +$ sin$^2(t) = 1$, making the following substitutions: 
$$ cos(t) = \frac{2w}{1+w^2}, sin(t) = \frac {1-w^2}{1+w^2},$$
and then taking the derivative and setting the dot product with a unit vector equal to zero, Kuiper obtains an equation of the form 
$$A^{4r}(w)(1+w^2)^{N-2r} + \epsilon B^{2N}(w) = 0,$$ 
where $A^{4r}$ and $B^{2N}$ are polynomials in $w$ with degree $4r$ and $2N$. We note that when $\epsilon = 0$, there are $N-2r$ roots of $i$ and $N-2r$ of $-i$, and thus there are at most $4r$ real roots when $\epsilon = 0$. Continuity ensures that for small $\epsilon > 0$, the number of real roots will not increase, and thus for some conformation of the $\textit{r-braid}$ knot there are at most $2r$ local maxima (since every maximum must have a corresponding minimum). This leads to the conclusion that $sb[K] \leq 2\beta[K]$. 

For our purposes, we will need the following.

\begin{lemma} \label{twocritical}
Given any nonzero vector direction $\vec{v} = v_1\hat{i}+v_2\hat{j}+v_3\hat{k}$, over the interval $t \in (0, \pi/2)$, the curve $\eta(t) = (\cos  t, \sin  t, \cos^2  t)$ has at most two critical points when projected to the real line defined by $\vec{v}$.  
\end{lemma}

\begin{proof} Take the derivative $\eta ' = (- \sin t, \cos t, -2 \sin t \cos t)$. Critical points occur when $\eta ' \cdot \vec{v} = 0$, which is to say $$-v_1 \sin t + v_2 \cos t -2 v_3 \sin t \cos t = 0. $$

Note that when $v_3 = 0$, we are projecting to vectors in the $xy$-plane. Since $\eta$ projects to a circle in the plane,  there are exactly two critical points on opposite sides of the circle for any such vector $\vec{v}$,  and at most one critical point for $0 < t < \pi/2$.

When $v_3 \neq 0$, we  obtain:

$$\frac{v_1}{2v_3} \sin t + \frac{-v_2}{2v_3} \cos t +  \sin t \cos t = 0$$.

Let $ a = \frac{v_1}{2v_3}$ and $b = \frac{-v_2}{2v_3}$. Then we have $a \sin t + b\cos t + \sin t \cos t = 0$. When $0 < t < \pi /2$, we can let $x = \sin t$ and $\sqrt{1-x^2} = \cos t $ where $0 < x < 1$.

Restating the problem now, we would like to show that the function 
$g_1(x)  = a x + (b + x)\sqrt{1-x^2}$ has at most two zeros for $0 < x < 1$ for all possible choices of real numbers $a$ and $b$.
We consider various possibilities for $a$ and $b$.

\noindent Case 1. $a= 0$. Then $x = -b$ is the only zero, which may or may not be in the interval (0,1), depending on the value of $b$.

\noindent Case 2. $b = 0$. Then $a = -\sqrt{1 - x^2}$ and $x = \sqrt{1 - a^2}$ is the only potential zero in $(0,1)$, and appearing as a zero depending on the value of $a$.

We now assume both $a$ and $b$ are nonzero. Taking $g_1(x) = 0$, moving the $ax$ to the other side of the equation and squaring yields  $$a^2 x^2 = b^2 + 2bx + (1 - b^2) x^2 - 2 b x^3 - x^4. $$ Thus every zero of $g_1$ is also a zero of $$f = b^2 +2bx + (1 - a^2 - b^2) x^2 - 2bx^3 - x^4.$$ 
Therefore there are at most four zeros of $g_1$ over all values of $x$. Define $g_2 = -ax + (b+x) \sqrt{1 - x^2}$. Then $f = g_1 \cdot g_2$,so  any zeros of $g_2$ are also zeros of $f$. Also, since $a \neq 0$, the zeros of $g_2$ are distinct from the zeros of $g_1$.

However, $g_2(-1) = a$ and $g_2(1) = -a$. Thus, $g_2$ has at least one zero and therefore $g_1$ has at most three zeros over all values of $x$. 
We now continue to consider cases.

\noindent Case 3. $a > 0$ and $b> 0$. Then clearly for $ 0 < x < 1$, all terms in $g_1$ are positive and there are no zeros.

\noindent Case 4. $a < 0$ and $b< 0.$ Then $g_1(-1) > 0$ and $g_1(0) < 0$, so $g_1$ has a zero in the $x$-interval $(-1, 0).$ Therefore it can have at most two zeros remaining for the interval (0,1).

\noindent Case 5. $a < 0$ and $b > 0.$ Let $h_1 = |a|x$ and $h_2 = (b+x)\sqrt{1-x^2}$. Then a zero of $g_1$ satisfies $h_1 = h_2$.
But $h_2' = \frac{1 -2x^2 -bx}{\sqrt{1-x^2}}$  which yields critical points at $\frac{-b \pm \sqrt{b^2 + 8}}{4}$. So there is only one maximum for positive $x$, and $h_2'' < 0$. Further $h_2(0) = b > 0$ . So the ray of slope $|a|$ defined by $h_1$ can only cross the graph of $h_2$ once for $x > 0$, and we have at most one zero of $g_1$ in the x-interval (0,1).

\noindent Case 6. $a > 0$ and $b< 0$. Let Let $j_1 = -ax$ and $j_2 = (b+x)\sqrt{1-x^2}$. Then a zero of $g_1$ satisfies $j_1 = j_2$.
But $j_2' = \frac{1 -2x^2 -bx}{\sqrt{1-x^2}}$, and again critical points occur at $\frac{b \pm \sqrt{b^2 + 8}}{4}$. Since $b < 0$, only one critical point occurs for $x > 0$, which is a maximum and $j_2(0) = b < 0$. Also $j_2(1) = 0$, so the ray given by $j_1 = -ax$ can cross the graph of $h_2$ at most once, and $g_1$ has at most one zero for $0 < x < 1$.
\end{proof}

In order to prove Theorem \ref{mainthm}, we utilize $n$-plats. An $n$-plat is constructed from an open braid with $2n$ strings, by pairing off the adjacent endpoints, left to right on the top and then also on the bottom and then gluing simple arcs with one local maximum/minimum to each pair of endpoints. Every $n$-bridge knot has a representation as an $n$-plat obtained by taking an $n$-bridge presentation and stretching all of the local maxima up to the same level and stretching down all of the local minima to the same level, increasing the number of crossings as necessary.

\begin{lemma}\label{freestrand}Given an $n$-plat representation of a knot or link, we can always free one strand, while preserving the fact we have an $n$-plat. 
\end{lemma}

\begin{proof}
In the braid portion of the $n$-plat representation, the leftmost string $s$, which starts at the top in the first position, ends at the bottom in some position $i$.
We can add crossings at the  bottom of the braid in order to move the string back to the left so that it also ends in the first position, still preserving the fact we have an $n$-plat representation of the same link. For any other string that it crosses, it must do so an even number of times. Then pull $s$ taut, so it appears as a vertical strand. Although this may create many additional crossings, we still have an $n$-plat representation. If some of the resulting strands to the left of the taut string are nested, as in Figure \ref{freestrandfig} (a),  we can fold them back so that the only regions to the left of $s$ that remain are un-nested bigons.

\begin{figure}[htbp]
	\centering
  \includegraphics[scale=.5]{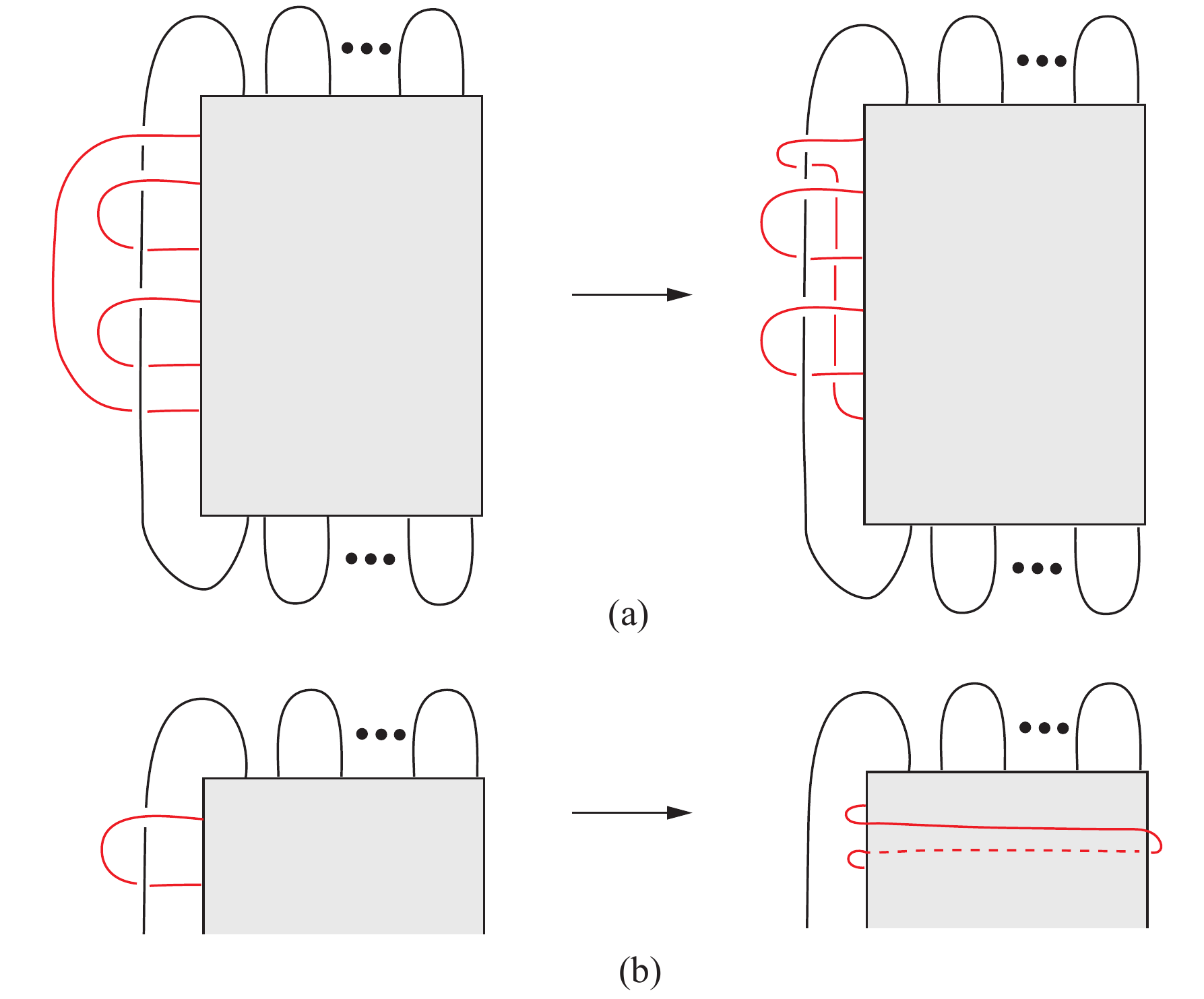}
    \caption{Freeing the leftmost strand in an $n$-plat.}
    \label{freestrandfig}
\end{figure}

Starting with the topmost such, we can lift the string making the bigon up over the top of the plat and down the other side, as in Figure \ref{freestrandfig}(b), removing the bigon while preserving the $n$-plat. Repeating with all bigons, we now have an $n$-plat presentation with a free strand on the left.
\end{proof}

\subsection*{The Construction}

Given any knot K with bridge index $n$, we begin with an $n$-plat projection $P$ in the $xy$-plane such that $P$ realizes the bridge index  of $K$, all local maxima occur at $y= 1$, all local minima occur at $y = 0$, and the strands travel from maxima to minima without inflection points with respect to the $y$-direction. By Lemma \ref{freestrand}, we free the leftmost strand. We call this strand $\textit{loose}$. Next we isotope our conformation such that it all lies within a small $\delta>0$ of the $xy$-plane, and such that the entire conformation other than the loose strand is confined to a $\delta$ neighborhood of the sub-arc of the curve $f(t)=(\cos(t), \sin(t))$ defined by $.1 < t < \pi/2 - .1$ The loose strand goes around the far side of the circle. See Figure \ref{baseballcurve1}.

We next proceed with a similar construction to \cite{FLS} by attaching two strands to each non-loose maximum such that the $i^{th}$ maximum is connected to the $i^{th}$ minimum with a point of singularity at the attachments. Furthermore, we can do this in such a way that the extra strands we have added wind around the $z$-axis with a height of zero in the $z$-direction and do not cross any other added strands or any part of the projection in the $xy$-plane and every added strand stays within a small $\epsilon$-neighborhood of the unit circle in the $xy$-plane defined by $f(t)=(\cos(t), \sin(t))$, as in the left side of Figure \ref{baseballcurve4}. We now have a $(2n-1)$-braid conformation of a singular knot $\hat{K}$ because the leftmost maximum of the original $n$-plat contributed one loose strand, and the other $n-1$ maxima each contribute two added strands. If we call the interior of the $2n-2$ added strands collectively $L$, then we note that $\hat{K}\backslash L$ is isotopic to our original knot $K$. As noted in \cite{FLS}, the singular points of attachment do not affect Kuiper's parameterization using functions $\lambda_1(t)$ and $\lambda_2(t)$ to find a conformation isotopic to $\hat{K}$. Since our conformation is currently within an $\epsilon$-neighborhood of the unit circle on the $x-y$ plane and within a small $\delta$ of the $x-y$ plane in the $z$-direction, we can change the $z$ coordinates to be within an $\epsilon$-neighborhood of the function $\cos^2(t)$, as shown on the right in Figure \ref{baseballcurve4}. Thus for $\epsilon > 0$, we have 
\begin{eqnarray*}
\hat{K}_\epsilon(t) &  = (\cos((2n-1)t)  
(1+\epsilon\lambda_1(t)), \sin((2n-1)t)(1+\epsilon\lambda_1(t)), \\  & 
\cos^2((2n-1)t)+\epsilon\lambda_2(t))
\end{eqnarray*}

where $\lambda_1^2 + \lambda_2^2 \leq 1$. We note that $\hat{K}_\epsilon(t)$ defines a $(2n-1)$-braid that sits within an $\epsilon$-wide tubular neighborhood of $\eta(t)$. Furthermore, the section of the curve containing the crossings associated with the knot $K$ lie in a region around $t=\pi/4$. This means most of the curve $\eta(t)$ is followed only by the original single loose strand and the added strands $L$. 

\begin{figure}[htbp]
	\centering
  \includegraphics[scale=.5]{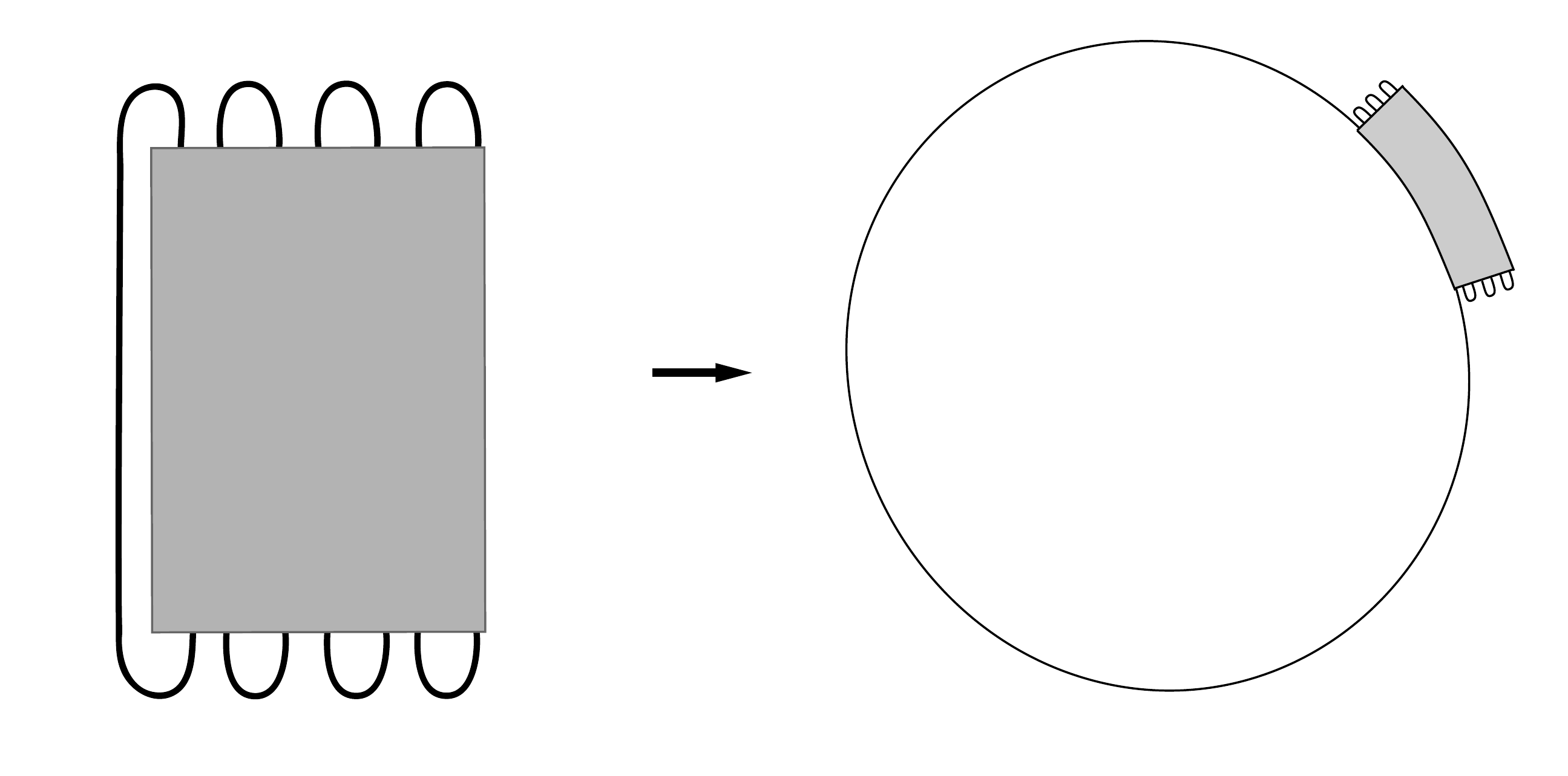}
    \caption{Istotoping an n-plat to within an $\epsilon$-neighborhood of the unit circle}
    \label{baseballcurve1}
\end{figure}

\begin{figure}[htbp]
	\centering
  \includegraphics[scale=.5]{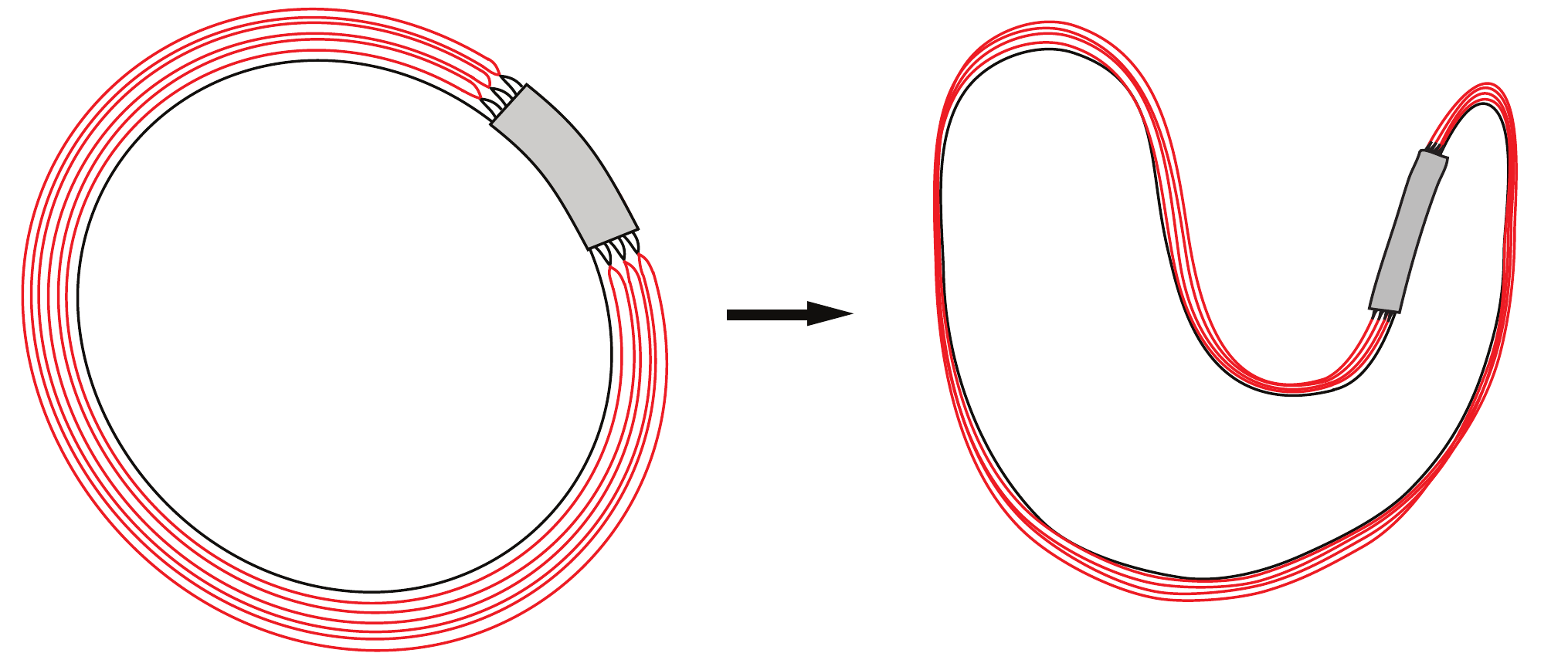}
    \caption{Adding the loose strands and placing on the curve $\eta$(t).}
    \label{baseballcurve4}
\end{figure}

\begin{proof} [Proof of Theorem \ref{mainthm}]

We will use $\hat{K}_\epsilon$ to obtain a bound on the superbridge number of our original knot $[K]$ by examining $sb(\hat{K}_\epsilon(t))$ and discounting the added strands (because the knot ends at the extrema where we added the extra strands). Let  $\vec{v} = v_1\hat{i}+v_2\hat{j}+v_3\hat{k}$ be a vector that defines the direction to which we project and let $J_\epsilon = \hat{K}_\epsilon \setminus L$. Let $E$ be the collection of $2n-2$ points on $J_\epsilon$ where the additional strands are attached, and call points on $J_\epsilon$ that are not in $E$ interior points.

We know that $\eta$ has at most two critical points in the direction of $\vec{v}$ in the arc defined by $ 0 < t < \pi/2$. When there are two critical points on the arc, since they are adjacent on the curve, at most one is a maximum. 

We can choose $\epsilon$ small enough that the variation in each strand due to the functions $\lambda_1$ and $\lambda_2$ is not greater than the curvature of the larger curve $\eta$. Thus, each string of the braid will have a critical point that is very close to the corresponding critical point on $\eta$.

For direction vectors that yield one maximum for $\eta$ and that maximum  is not in the interval $0 < t < \pi/2$, there are no maxima at the interior points of $J_\epsilon$ excluding the loose strand. However, if there is a minimum in the interval $0 < t < \pi/2$, then all of the points in $E$ can be maxima for $J_\epsilon$. In addition the loose strand has another maximum. So in this case, the total number of maxima can be at most $(2n-2) + 1= 2n -1$.

For direction vectors that yield two maxima for $\eta$, both of which are not in the arc of $\eta$ given by $0 < t < \pi/2$, we know that each of the $2n-2$ added strands and the one original loose strand contribute 2 local maxima. However, we can ignore the maxima contributed by all the $2n-2$ added strands, and instead view the $2n-2$ points in $E$ as the only other potential maxima. This leaves us with at most $2n-2 + 2 = 2n$ total maxima since the original loose strand contributes two maxima. 

For a direction vector that has a maximum in the arc on $\eta$ corresponding to the interval $0<t < \pi/2$, Lemma \ref{twocritical} limits us to at most two critical points. When there are two critical points, at most one can be a maximum.  If all the critical points on the individual strings occur in the region containing the $n$-plat, each of the $2n-1$ strands of the $n$-plat contributes one maximum. Each strand also contains a minimum, which means that the corresponding  $n-1$ singular points in $E$ where we glued on the additional strands will appear as maxima of the original knot. The loose strand will also have a potential maximum on it, so we have a total of $(2n-1) + (n-1) + 1 = 3n-1$ possible maxima.

If there is only one maximum and no minimum on the arc corresponding to $0 < t < \pi/2$, we have at most $(2n-2) + 1$ maxima, the last coming from the loose strand.

When critical points occur in the interval corresponding to $0 < t < \pi/2$ but do not necessarily correspond to critical points on $J_\epsilon$, we must be careful about the transition of critical points around the singular points in $E$. As we change our direction vector $\vec{v}$, we know that the region around each singular point will resemble Figure \ref{localmaxima}. Each string of the singular braid can contribute at most one local maximum in this region. As we vary the projection vector and see the two maxima corresponding to a pair of strands in $J_\epsilon$ that share that singular point moving to the singular point, we stop counting each at the instant the maximum coincides with the singularity. The first to pass through the singularity does not cause a change in the count of maxima for $J_\epsilon$ since it is replaced by the maximum at the singularity. The second maximum, when reaching the singularity disappears as a maximum for $J_\epsilon$. 
Thus, the total number of maxima does not go up when we transition maxima out of the interior of the strands of $J_\epsilon$. 

\end{proof}

\begin{figure}[htbp]
	\centering
  \includegraphics[scale=.3]{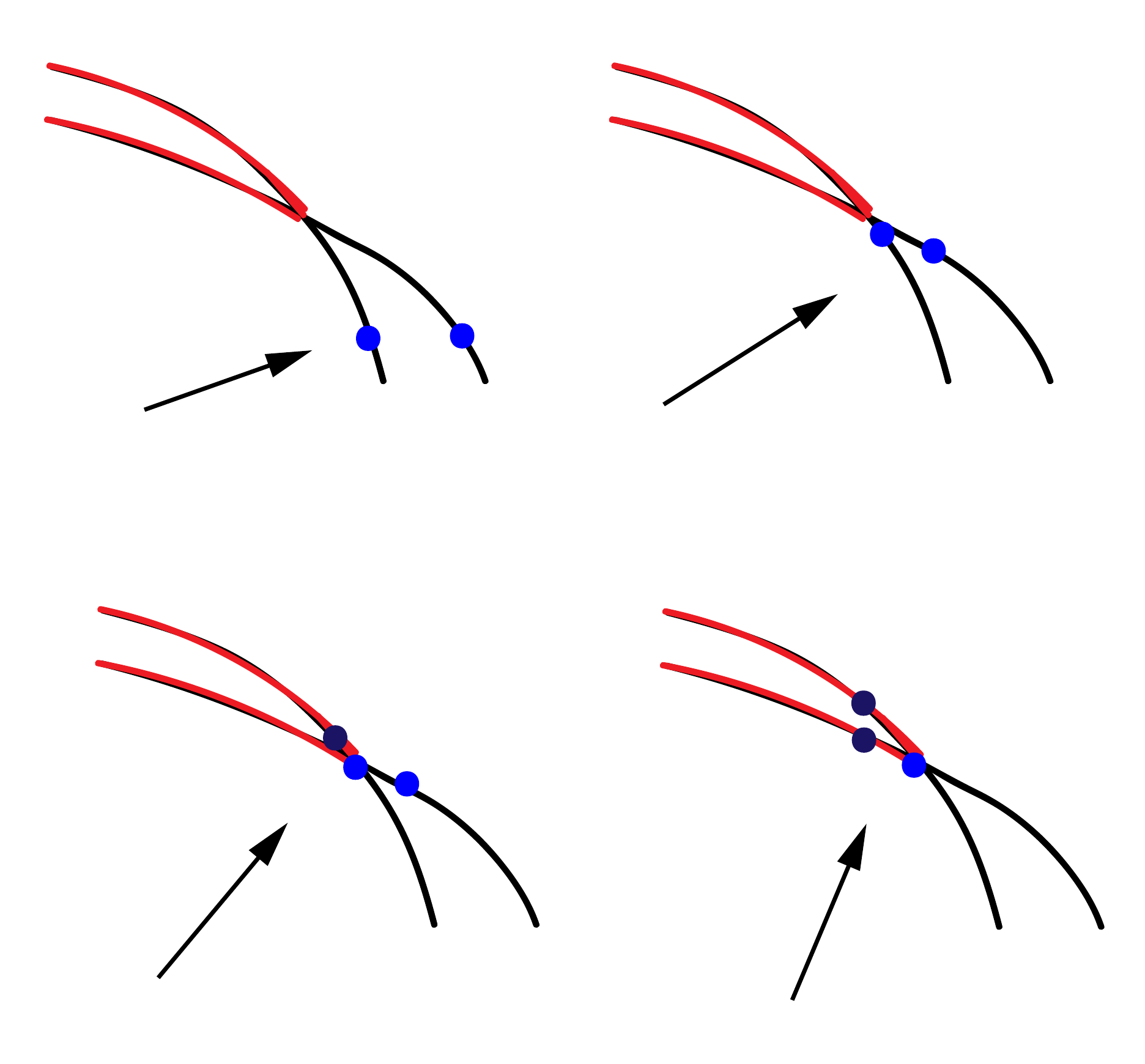}
    \caption{The local maxima at the singularity points as $\vec{v}$ rotates counterclockwise. Each strand contributes zero or one maxima.}
    \label{localmaxima}
\end{figure}

{\footnotesize COLIN ADAMS, DEPARTMENT OF MATHEMATICS, WILLIAMS COLLEGE, WILLIAMSTOWN, MA 01267 \\
{\it E-mail address:} {\bf cadams@williams.edu}

NIKHIL AGARWAL, DEPARTMENT OF ECONOMICS, MASSACHUSETTS INSTITUTE OF TECHNOLOGY, CAMBRIDGE, MA 02139 \\
{\it E-mail address:} {\bf agarwaln@mit.edu}

RACHEL ALLEN, CIVIL AND ENVIRONMENTAL ENGINEERING, UNIVERSITY OF CALIFORNIA, BERKELEY, CA 94720\\
{\it E-mail address:} {\bf rachelallen@berkeley.edu}

TIRASAN KHANDHAWIT, MAHIDOL UNIVERSITY, BANGKOK, THAILAND\\
{\it E-mail address:} {\bf tirasan@gmail.edu}

ALEX SIMONS, DEPARTMENT OF MATHEMATICS, WILLIAMS COLLEGE, WILLIAMSTOWN, MA 01267\\
{\it E-mail address:} {\bf ads4@williams.edu}

REBECCA WINARSKI, COLLEGE OF THE HOLY CROSS, 1 COLLEGE ST.,
WORCESTER, MA 01610\\
{\it E-mail address:} {\bf rwinarski@holycross.edu}

MARY WOOTTERS, DEPARTMENTS OF COMPUTER SCIENCE3 AND ELECTRICAL ENGINEERING, STANFORD UNIVERSITY, STANFORD, CA 94304\\
{\it E-mail address:} {\bf marykw@stanford.edu}}


\begin{thebibliography}{99}

\bibitem{FLS} E. Furstenberg, J. Lie, and J. Schneider, {\it Stick knots}, Chaos, Solitons, and
Fractals 9(4/5) (1998) 561-568.

 \bibitem{JJ1} C. B. Jeon, G.T. Jin, {\it A computation of superbridge index of knots} J. Knot Theory Ramifications, Vo.11 No. 3(2000) 461--473.
 
   \bibitem{JJ2} J. B. Jeon, G. T. Jin, {\it There are only finitely many 3-superbridge knots}, Knots in Hellas '98, Vol. 2 (Delphi). 
J. Knot Theory Ramifications 10 (2001), no. 2, 331-343. 


 \bibitem{Jin1} G. T. Jin, {\it Polygon indices and superbridge indices of torus knots and links} J. Knot Theory Ramifications 6 (1997), no. 2, 281-289.
 
 \bibitem{Jin2} G. T. Jin, {\it Superbridge index of composite knots}  J. Knot Theory Ramifications 9 (2000), no. 5, 669-682. 

\bibitem{Jin3}G. T. Jin, {\it  Superbridge index of knots} Kobe J. Math., Vo.18 No. 2(2001) 181--197.

\bibitem{Kuip} N. Kuiper, {\it A new knot invariant}, Math. Ann. 278 (1987) 193-209.

\bibitem{Milnor} J. Milnor, {\it On the total curvature of knots}, Annals of Math. 52 (1950) 248-257.

\bibitem{Nak} Y. Nakanishi, {\it Primeness of links} 
Math. Sem. Notes Kobe Univ. 9 (1981), no. 2, 415-440. 

\bibitem{Sch} H. Schubert, {\it Uber eine numerische knoteninvariante},
Math. Z. 61 (1954), 245-288. 



\end{thebibliography}
\end{document}